\documentclass{alggeom}
\usepackage{amsmath}
\usepackage{MnSymbol}
\usepackage[all]{xy}
\usepackage{color}
\usepackage{graphicx}
\usepackage{enumerate}

\usepackage{ifthen}
\newcommand{\cit}[1]{{\rm \textbf{#1}}}

\newcommand{\Ref}[2]{\cit{%
\ifthenelse{\equal{#1}{thm}}{Theorem}{}%
\ifthenelse{\equal{#1}{prop}}{Proposition}{}%
\ifthenelse{\equal{#1}{lem}}{Lemma}{}%
\ifthenelse{\equal{#1}{cor}}{Corollary}{}%
\ifthenelse{\equal{#1}{defn}}{Definition}{}%
\ifthenelse{\equal{#1}{oss}}{Remark}{}%
\ifthenelse{\equal{#1}{rmk}}{Remark}{}%
\ifthenelse{\equal{#1}{sec}}{Section}{}%
\ifthenelse{\equal{#1}{ex}}{Example}{}%
\ifthenelse{\equal{#1}{conj}}{Conjecture}{}%
\ifthenelse{\equal{#1}{ssec}}{Subsection}{}%
\ifthenelse{\equal{#1}{tab}}{Table}{}%
\ifthenelse{\equal{#1}{cla}}{Claim}{}%
\  \ref{#1:#2}%
}}
\usepackage{hyperref}

\newtheorem{prop}{Proposition}[section]
\newtheorem*{prop*}{Proposition}
\newtheorem{thm}[prop]{Theorem}
\newtheorem*{thm*}{Theorem}
\newtheorem{lem}[prop]{Lemma} 

\newtheorem*{cor*}{Corollary}

\theoremstyle{definition}

\theoremstyle{remark}
\newtheorem{oss}[prop]{Remark}
\newtheorem{rmk}[prop]{Remark}

\numberwithin{equation}{section}
\newcommand{\hk}{hyperk\"{a}hler }

\newcommand{\kahl}{K\"{a}hler }

\newcommand{\kntiposp}{$K3^{[n]}$-type }

\newcommand{\issp}{irreducible symplectic }

\DeclareMathOperator{\Ext}{Ext}

\DeclareMathOperator{\ch}{ch}

\newcommand{\id}{\mathrm{id}}

\newcommand{\mc}[1]{\mathcal{#1}}

\newcommand{\Z}{\mathbb{Z}}
\newcommand{\C}{\mathbb{C}}

\newcommand{\PP}{\mathbb{P}}

\DeclareMathOperator{\td}{td}
\newcommand{\aut}{\mathrm{Aut}}

\begin{document}

\title{Automorphisms of O'Grady's Manifolds Acting Trivially on Cohomology}
\author{Giovanni Mongardi}
\email{giovanni.mongardi@unimi.it}
\address{Department of Mathematics, University of Milan\\ via Cesare Saldini 50, Milan}
\author{Malte Wandel}
\email{wandel@kurims.kyoto-u.ac.jp}
\address{Research Institute for Mathematical Sciences,
Kyoto University, Kitashirakawa-Oiwakecho, Sakyo-ku, Kyoto, 606-8502 Japan}
%
\thanks{The first named author was supported by FIRB 2012 ``Spazi di Moduli e applicazioni''.\\
The second named author was supported by JSPS Grant-in-Aid for Scientific Research (S)25220701 and by the DFG research training group GRK 1463 (Analysis, Geometry and String Theory).}

\begin{abstract}
We determine the subgroup of automorphisms acting trivially on the second integral cohomology for  \hk manifolds which are deformation equivalent to O'Grady's sporadic examples. In particular, we prove that this subgroup is trivial in the ten-dimensional case and isomorphic to $(\Z/2\Z)^{\times 8}$ in the six-dimensional case.
\end{abstract}

\maketitle
\addcontentsline{toc}{section}{Abstract}
\vspace*{6pt}
\section*{Introduction}\label{sec:introduction}
\addcontentsline{toc}{section}{Introduction}
Automorphisms of \issp or \hk manifolds have recently been studied by numerous mathematicians pursuing varying objectives and using different techniques. A key to the understanding of the underlying geometry is usually played by understanding the induced action of an automorphism on the second integral cohomology: For any \issp manifold $X$ the cohomology group $H^2(X,\Z)$ carries a natural non-degenerate lattice structure and a weight-two Hodge structure. An automorphism of $X$ preserves both these structures and we obtain a homomorphism of groups:
\[\nu\colon \aut(X)\rightarrow O(H^2(X,\Z)).\]
It is this homomorphism that allows us to study the geometry of $X$ and its automorphisms using lattice theory. This has been done very successfully and extensively in the case of K3 surfaces. From the strong Torelli theorem for K3 surfaces it follows that $\nu$ is injective in this case. Thus, by passing from the geometric picture on to the lattice side, we do not lose any information and we can classify automorphisms using lattice theory.

When constructing the first examples of higher-dimensional \issp manifolds, Beauville (cf.\ \cite{Beau83}[Prop.\ 10]) soon realised that the injectivity of $\nu$ holds also true in the case of Hilbert schemes of points on a K3. This fact was applied by Boissi\`{e}re$-$Sarti (\cite{BS12}) to understand when an automorphism on such a Hilbert scheme is induced by an automorphism of the surface, constituting a first step in the classification of automorphisms of Hilbert schemes. These results should not lead to the idea that $\nu$ is injective in general. It was shown in \cite{BNS11} that for generalised Kummer varieties of dimension $2n-2$ the kernel of $\nu$ is generated by induced automorphisms coming from the underlying abelian surface, i.e. by translations by
points of order $n$ and by $-\id$. These automorphisms preserve the Albanese fibres of the Hilbert scheme of $n$ points on the surface and they surely act trivially on cohomology. Thus, in general, $\ker\nu$ is not trivial. But in the case of generalised Kummer varieties we at least understand the action of the kernel very well. Thus it is still possible to use lattice theory for an understanding of the automorphism group.

A fundamental step towards a better understanding of the kernel of $\nu$ is the result by Hassett$-$Tschinkel stating that - as a group - this kernel is a deformation invariant of the manifold $X$ (cf.\ \cite{HT13}[Thm.\ 2.1]). It implies that we now know the group structure of $\ker\nu$ for all manifolds of \kntiposp and of Kummer-type. Note, that for a general deformation of a generalised Kummer variety we do not have an explicit construction of the automorphisms in $\ker\nu$.

There are two more known deformation types of \hk manifolds. The first examples in both cases have been constructed by O'Grady (cf.\ \cite{OGr99} for the ten-dimensional example and \cite{OGr03} for the example of dimension six). The main results of this article concern the kernel of the cohomological representation $\nu$ for manifolds which are deformation equivalent to these manifolds.

In particular, we prove the following two theorems:
\begin{thm*}[(\Ref{thm}{thm_og10})]
Let $X$ be a manifold of $Og_{10}$-type. Then $\nu$ is injective.
\end{thm*}

\begin{thm*}[(\Ref{thm}{thm_og6})]
Let $X$ be a manifold of $Og_6$-type. Then
\[\ker\nu\cong(\Z/2\Z)^{\times 8}.\]
\end{thm*}


Let us outline the idea of the proofs of the theorems. The geometry of O'Grady's examples is more complicated and less understood compared to the case of Hilbert schemes of points or Generalised Kummers. Therefore a detailed analysis was needed. In the ten-dimensional case we consider a relative compactified Jacobian of degree four over a (five-dimensional) linear system of genus five curves on a K3 surface. Its resolution of singularities is an \issp manifold of $Og_{10}$-type. There are two main ingredients to the proof of the injectivity of $\nu$. First we show that every automorphism in $\ker\nu$ acts fibrewise on the Jacobian. Secondly we prove that the relative Theta divisor is rigid, thus must be preserved by any such automorphism. For the second step (rigidity of the relative Theta divisor) we first prove that the relative Theta divisor has the structure of a $\PP^1$-bundle. From this we deduce its rigidity using a criterion that should be known to the experts (cf.\ \Ref{lem}{lem_rigid}). We can then conclude that $\ker\nu$ is trivial using the Torelli Theorem for Jacobians.

In the six-dimensional case we follow a similar idea, where this time the first step (fibrewise action) turns out to be easier to prove.  Note that the automorphism induced by $-\id$ on the abelian surface acts trivially on $Og_6$. The group $\ker\nu$ is composed by automorphism induced by translation by two-torsion points on $A\times A^*$, where $A$ is the underlying abelian surface and $A^*$ its dual.

As a last result we study the automorphisms in $\ker\nu$ (in the six-dimensional case) more in detail. In particular, we study their fixed loci and conclude that it either consists of $16$ disjoint K3 surfaces ($30$ cases, see \Ref{prop}{fixed_ta} and \Ref{prop}{fixed_l}), of $2$ disjoint $K3$ surfaces ($45$ cases) or $16$ isolated fixed points (the remaining $180$ cases). (For the last two situations see \Ref{prop}{fixed_mixed}.)

We want to emphasise one more result -- since it is a beautiful result on its own -- which is a side product of the proof of the injectivity of $\nu$ for $Og_{10}$:
\begin{prop*}
Let $S$ be a $K3$ surface being a double cover of $\mathbb{P}^2$ ramified along a sextic curve admitting a unique tritangent. Denote by $H$ the pullback of $\mc{O}(1)$.
Then the rational map $|2H|\dashedrightarrow \overline{\mc{M}_5}$ is injective.
\end{prop*}
See \Ref{prop}{injective_curves} for details.

The motivation to prove these results and to write this article grew out of the attempt to detect (using lattice theory) certain induced automorphisms (on O'Grady's manifolds) that have been constructed by the authors in \cite{MW14}. Using the article at hand we managed to state a lattice theoretic criterion to detect induced automorphisms on O'Grady-type moduli spaces (cf. loc.\ cit.\ Prop.\ 4.5).

The structure of this article is as follows: We gather the required background on the representation $\nu$ and on O'Grady's manifolds in \Ref{sec}{prel}. Then we continue by fully treating the ten-dimensional case in \Ref{sec}{og10}. Before continuing with the study of $\nu$ in the six-dimensional case in \Ref{sec}{og6} we first study the geometry of a special six-dimensional O'Grady manifold in more detail in the second preliminary \Ref{sec}{pre6}. We conclude with the study of the fixed point loci in the final \Ref{sec}{fixed}.

\section*{Acknowledgements}
\addcontentsline{toc}{section}{Acknowledgements}
The first named author would like to thank GNSAGA/INdAM for its support and the research Institute of mathematical sciences of Kyoto University for their hospitality. Moreover he would also like to thank Alice Garbagnati, Arvid Perego and Antonio Rapagnetta for useful discussions. 
 
The second named author wants to thank Gilberto Bini for his re-approved kind hospitality and for supporting his visit to Milano. Furthermore he wants to thank Jesse Kass and Keiji Oguiso for their inspiring suggestions and help and Hisanori Ohashi and Shigeru Mukai for many interesting discussions.

Both authors want to thank Kieran O'Grady for his kind help and Bert van Geemen for an inspiring discussion.

Finally, we are both grateful to the Max Planck Institut f\"ur Mathematik Bonn, for having partially supported the first named author and hosted the second named author.  

\section{Preliminaries}\label{sec:prel}
In this introductory section we will gather the most important background material and known results about the desingularised moduli spaces as introduced by O'Grady.

Let $S$ be a projective $K3$ or abelian surface. Mukai defined a lattice structure on $\widetilde{H}^*(S,\Z):=H^{*,ev}(S,\Z)$ by setting
\[(r_1,l_1,s_1).(r_2,l_2,s_2):=l_1\cdot l_2-r_1s_2-r_2s_1,\]
where $r_i\in H^0,$ $l_i\in H^2$ and $s_i\in H^4$. This lattice is referred to as the \em Mukai lattice \em and we call vectors $v\in \widetilde{H}^*(S,\Z)$ \em Mukai vectors. \em The Mukai lattice is isometric to $U^4\oplus E_8(-1)^2$ if $S$ is a $K3$ and $U^4$ if $S$ is abelian.

Furthermore we may introduce a weight-two Hodge structure on $\widetilde{H}^*(S,\Z)$ by defining the $(1,1)$-part to be
\[H^{1,1}(S)\oplus H^0(S)\oplus H^4(S).\]

For an object $\mc{F}\in D^b(S)$ we define the \em Mukai vector of $\mc{F}$ \em by
\[v(\mc{F}):=\ch(\mc{F})\sqrt{\td_S}.\]

The following theorem summarises the famous results about moduli spaces of stable sheaves on K3 surfaces.

\begin{thm}
Let $S$ be a projective K3 and let $v$ be a primitive Mukai vector. Assume that $H$ is $v$-generic. Then the moduli space $M(v)$ of stable sheaves on $S$ with Mukai vector $v$ is a \hk manifold which is deformation equivalent to the Hilbert scheme of $n$ points, where $2n=v^2+2$. Furthermore we have an isometry of lattices
\begin{equation}
H^2(M(v),\Z)\cong v^\perp\subset \widetilde{H}^*(S,\Z) \label{eq_hodge_isom}
\end{equation}
 which preserves the weight-two Hodge structures.
\end{thm}
\begin{proof}
\cite[Thm.\ 1.2]{PR14}.
\end{proof}

O'Grady studied a particular case of a non-primitive Mukai vector: Let $v\in \widetilde{H}^*(S,\Z)$ be a primitive Mukai vector of square $2$.

\begin{thm}[O'Grady, Perego$-$Rapagnetta]
The moduli space $M(2v)$ is a $2$-factorial symplectic variety of dimension ten admitting a Beauville$-$Bogomolov form and a pure weight-two Hodge structure on $H^2(M(2v),\Z)$ such that equation (\ref{eq_hodge_isom}) holds. Furthermore it admits a symplectic resolution $\widetilde{M}(2v)$ which is a \hk manifold that is not deformation equivalent to any of the known examples.
\end{thm}
\begin{proof}
\cite[Thm.\ 1.6 and Thm.\ 1.7]{PR13} and \cite[Thm.\ 1.1]{PR14}.
\end{proof}

In the case of abelian surfaces there is one more step to take:

Let $A$ be an abelian surface and $v$ a primitive Mukai vector. Then the moduli space $M(v)$ is not simply connected. Indeed, for a sheaf $\mc{F}$ we define $alb(\mc{F}):=(\Sigma c_2(\mathcal{F}),\det(\mathcal{F}))\in A\times A^*$ yielding an isotrivial surjective map $alb\,\colon M(v)\rightarrow A\times A^\vee$ which turns out to be the Albanese map of $M(v)$. The fibre $K(v):=alb^{-1}(0,0)$ is a \hk manifold and equation (\ref{eq_hodge_isom}) holds, if we compose the left hand side with the restriction to the fibre.

Finally we can also consider non-primitive Mukai vectors as above. So let us fix $v$ with $v^2=2$. Then we have a commuting diagram of resolutions and Albanese fibres:

\[\xymatrix{
\widetilde{K}(2v)\ar[r] \ar[d]&\widetilde{M}(2v)\ar[d] \ar[dr]^{alb}\\
K(2v)\ar[r] & M(2v)\ar[r]_{alb} & A\times A^\vee,
}\]
where $\widetilde{K}(2v)$ is a six dimensional \hk manifold.\\
\vspace{10pt}\\
Next, we state two fundamental results concerning automorphisms of \issp manifolds acting trivially on cohomology. Let $X$ be an \issp manifold. We let
\[\nu\colon \aut(X)\rightarrow O(H^2(X,\Z))\]
be the cohomological representation.
\begin{prop}[Huybrechts]\label{prop:prop_huy}
The kernel of $\nu$ is finite.
\end{prop}
\begin{proof}
\cite[Prop.\ 9.1]{Huy99}
\end{proof}

\begin{thm}[Hassett$-$Tschinkel] \label{thm:thm_HT}The kernel of $\nu$ is a deformation invariant of the manifold $X$.
\end{thm}
\begin{proof}
\cite[Thm.\ 2.1]{HT13}
\end{proof}
\vspace{10pt}
Finally we include a result concerning rigid divisors on symplectic varieties.
\begin{lem}\label{lem:lem_rigid}
Let $X$ be a $\mathbb{Q}$-factorial symplectic variety and let $D\subset X$ be a prime divisor admitting a fibration $g\colon D\rightarrow Z$ with generic fibre isomorphic to $\PP^1$. Then $D$ is rigid, i.e.\ $h^0(X,\mc{O}(D))=1$.
\end{lem}
\begin{proof}
Since $X$ has trivial canonical bundle, it is enough to prove that $H^0(D,K_D)=0.$  But the latter is isomorphic to $H^0(Z,g_*K_D)$ and the restriction of $K_D$ to a fibre of $g$ is $\mc{O}_{\PP^1}(-1)$, thus $g_*K_D$ is $0$.
\end{proof}

\section{The Ten-Dimensional Case}\label{sec:og10}
In this chapter we will prove the following theorem:

\begin{thm}\label{thm:thm_og10}
Let $X$ be a manifold of $Og_{10}$-type. Then the cohomological representation
\[\nu\colon \aut(X)\rightarrow O(H^2(X,\Z))\]
is injective.
\end{thm}
\begin{proof}
By 
\Ref{thm}{thm_HT} the kernel of $\nu$ is a deformation invariant, thus we may assume that $X$ is the desingularisation of the moduli space $M(2v)$, $v=(0,H,2)$ on a $K3$ surface $S$ which is a double cover of $\PP^2$ branched along a sextic curve $\Gamma$ and where $H$ denotes the pullback of $\mc{O}(1)$. That is, $X$ is the desingularisation of the relative compactified Jacobian $M(0,2H,4)=\mc{J}^4(|2H|)$ of degree four over $|2H|$ and it comes with a lagrangian fibration $X\rightarrow |2H|$ which factors as the blow down followed by the map $\pi\colon \mc{J}^4(|2H|)\rightarrow |2H|$ assigning to a sheaf its support. We may choose the sextic $\Gamma$ as follows. Let $\Gamma^*$ be a plane quartic with an ordinary triple point. Its dual curve $\Gamma$ is a sextic curve with a unique tritangent.
Now, let $\psi$ be an automorphism of $X$ acting trivially on $H^2$. 
Let us prove that $\psi=\id$. First of all, $\psi$ fixes the class of the exceptional divisor of the blow up $X\rightarrow \mc{J}^4(|2H|)$. Thus the automorphism descends to an automorphism $\psi'$ of the singular relative Jacobian $\mc{J}^4(|2H|)$ still acting trivially on second cohomology.

\begin{lem}\label{lem:lem_og10_rigid}
The relative theta divisor $\Theta_{|2H|}$ is an effective rigid divisor on $\mc{J}^4(|2H|)$.
\end{lem}


\begin{proof}
We will use \Ref{lem}{lem_rigid}. Let $C$ be a general curve in $|2H|$. The fibre $\pi^{-1}(C)$ is isomorphic to the Jacobian $\mc{J}^4(C)$. The Theta divisor $\Theta_C$ is given as
\[ \{ \mc{O}(p_1+\cdots +p_4) \mid p_1,\dots,p_4\in C\}.\]
We can therefore define a rational map
\begin{eqnarray*}
\Theta_{|2H|}&\dashedrightarrow& Sym^4S,\\
\mc{O}(p_1+\cdots p_4)&\mapsto & p_1+\cdots +p_4.
\end{eqnarray*} 
The general fibre of this map can be identified with the set of curves in $|2H|$ that pass through four given points. These are four linear conditions cutting out a line.
\end{proof}

Since the class of the pullback $\pi^*\mc{O}(1)$ is fixed by $\psi'$, we see that $\pi$ is $\psi'$-equivariant, that is, $\psi'$ maps fibres of $\pi$ to fibres. Since generically these fibres are Jacobians of smooth curves and -- by the above lemma -- the classes of the respective theta divisors are mapped to each other, the Torelli theorem for Jacobians yields an isomorphism of the underlying curves. We continue by showing that this already implies that $\psi'$ acts fibrewise. We will therefore prove the following result which is interesting on its own.

\begin{prop}\label{prop:injective_curves}
Let $S$ be a $K3$ surface being a double cover of $\PP^2$ ramified along a sextic curve that admits a unique tritangent. Denote by $H$ the pullback of $\mc{O}(1)$.
Then the rational map $\varphi\colon|2H|\dashedrightarrow \overline{\mc{M}_5}$ is injective.
\end{prop}

\begin{proof}

First we note that the differential of the map is injective. This can be seen as follows: Let $C\in |2H|$ be a stable curve. The differential of $\varphi$ at the point corresponding to $C$ is given as the coboundary map
\[ H^0(\mc{N}_{C|S}) \rightarrow H^1(\mc{T}_C)\]
in the long exact cohomology sequence associated with the normal bundle sequence
\[ 0\rightarrow \mc{T}_C \rightarrow \mc{T}_S|_C \rightarrow \mc{N}_{C|S}\rightarrow 0.\]
Thus it is enough to prove $h^0(\mc{T}_S|_C)=0.$ This can be done using the same method as in the second half of the proof of Proposition 1.2 in \cite{CK13}. The rational map $\varphi$ has an indeterminacy locus of codimension at least two and can be extended to a morphism from a suitable blow up of $|2H|$. Thus it is enough to prove injectivity along a divisor which is saturated in the fibres (i.e.\ a divisor $D$ such that $\varphi^{-1}(\varphi(D))=D)$. We will choose this divisor to be the symmetric square $S^2|H|\subset |2H|$ corresponding to reducible curves. The rational map $\varphi|_{S^2|H|}$ is given as the symmetric square of the map $|H|\dashedrightarrow \overline{\mc{M}}_2$. Thus we have reduced the problem to showing that the latter map is injective. Again, by blowing up $|H|$ we can extend this map to a proper morphism $\widetilde{|H|}\rightarrow \overline{\mc{M}}_2$. Furthermore, the discussion in Chapter 3C of \cite{HM98} shows that the exceptional locus in $\widetilde{|H|}$ is mapped to the locus in $\overline{\mc{M}}_2$ of curves having an elliptic tail and thus its image is disjoint from the image of the locus of stable curves in $|H|$. Hence it is enough to prove injectivity in a single point (inside the stable locus). Since we assumed that $\Gamma$ is a sextic with a unique triple tangent, we will choose this single point to correspond to the double cover $C_0$ of this triple tangent. But now the uniqueness of this triple tangent ensures that $C_0$ (as a member of $|H|$) is unique in its isomorphism class.
\end{proof}

Thus $\psi'$ acts fibrewise and fixes the class of the theta divisor. Since the divisor $\Theta_{|2H|}$ is rigid, $\psi'$ cannot be given by translations on the fibres. Thus, again by the Torelli theorem for Jacobians, we see that $\psi'$ acts trivially on all fibres corresponding to smooth curves, that is, $\psi'$ is the identity.
\end{proof}

\section{Preliminaries for the Six-Dimensional Case}\label{sec:pre6}
In this section we give a detailed description of the geometry of a specific example of a moduli space whose Albanese fibre is a manifold of $Og_6$-type. We will use this description in the next section, and then later also in \Ref{sec}{fixed} to study the fixed locus of the automorphisms acting trivial on cohomology.

As in the ten dimensional case we will consider a relative compactified Jacobian over a non-primitive linear system, i.e.\ we start by considering a moduli space of sheaves $M(2v)$ with Mukai vector $v$ of the form $(0,H,a)$ for an effective divisor class $H$ on an abelian surface $A$ and an integer $a$. Such a moduli space comes with a fibration $\pi\colon M(2v)\rightarrow \{2H\}$ over the continuous system $\{2H\}$. Note that in the case of abelian surfaces the $A^*$-component of the Albanese map $M(2v)\rightarrow A\times A^*$ factors via $\pi$ and the natural isotrivial fibration $\{2H\}\rightarrow A^*$ with fibre the linear system $|2H|$. Thus we will work with the partial Albanese fibre over a point in $A^*$ which can be identified with the relative compactified Jacobian $\mc{J}^{a+2H^2}(|2H|)$. (The arithmetic genus of the curves in $|2H|$ is $2H^2+1$.)

Thus let $\Gamma$ be a generic curve of genus two and denote by $(A,H)$ its Jacobian together with its principal polarisation given by a symmetric theta divisor. For the Mukai vector $v=(0,H,0)$ we obtain the Jacobian $\mc{J}^4(|2H|)$. In the following we will study the fibres of the restricted fibration map $\pi\colon \mc{J}^4(|2H|) \rightarrow |2H|$ according to the stratification of the linear system $|2H|$ and finally the fibre of the restriction of the Albanese map.

An important tool in the study of the relative Jacobian is the following classical observation: The linear system $|2H|$ on $A$ defines a degree two map $f\colon A\rightarrow |2H|^\vee\cong \PP^3$. The image is the singular (quartic) Kummer surface $Kum_s(A)$. If $C$ is a curve in $|2H|$, then its image $f(C)$ is a quartic curve in $\PP^3$ (of arithmetic genus $3$). We denote by the same symbol $\iota$ the involution $-\id_A$ on $A$ and its restriction to $C$ (which is the covering involution of $C$ over $f(C)$). The surface $Kum_s(A)$ is projectively self-dual (as a quartic in $\PP^3$).

We start by recalling Rapagnetta's result (cf.\ \cite[Prop.\ 2.1.3]{Rap07}) on the stratification of $|2H|$. The strata form quasi-projective subvarieties of $|2H|$; thus we will indicate their relation to the dual of $Kum_s(A)$.

\begin{prop}
Let $C$ be a curve in $|2H|$ then $C$ belongs to one of the following strata:
\vspace{5pt}\\
\begin{tabular}{|c|c|c|}\hline
Stratum & singularity type& geometry of stratum\\\hline\hline
S& smooth& open\\\hline
N(1)&$1$ node in $A[2]$& $16$ planes\\\hline
N(2)&$2$ nodes in $A[2]$ & $120$ lines\\\hline
N(3)&$3$ nodes in $A[2]$ & $240$ points\\\hline
R(1)&reducible with two connecting nodes& $Kum_s(A)^\vee$\\\hline
R(2)&reducible with one connecting cusp& dual of $f(H)$ (trope in $Kum_s(A)$)\\\hline
D& double curve & $16$ points (nodes of $Kum_s(A)$)\\
\hline
\end{tabular}
\vspace{5pt}\\
All curves in the strata $R(1)$, $R(2)$ and $D$ are of the form $H_x\cup H_{-x}$, where $H_x$ denotes the translate of the theta divisor $H$ by $x\in A$.
\end{prop}

Let us continue by studying the fibres of $\pi$, i.e.\ the compactified Jacobians of the curves $C$ according to the strata above.

If $C$ is a stable curve, i.e.\ belonging to $S$ or $N(i)\setminus \big{(}N(i)\cap R(1)\big{)}$, then the $\pi^{-1}(C)$ has a stratification given by the partial normalizations of $C$ in the nodes. In particular, the open stratum of $\pi^{-1}(C)$ is fibred over the Jacobian $\mc{J}^4(\widetilde{C})$ of the normalization where the fibre is a product of copies of $\mathbb{C}^*$, one for every node of $C$. For further details about compactified Jacobians of stable curves we refer to the summary in \cite[Sect.\ 4.1]{Cap09} or to \cite{Kas08} for a more detailed introduction (also for cuspidal curves).

If $C$ is in $ \big{(}R(1)\cup R(2) \big{)}\setminus N(1)$, i.e.\ the union of the two distinct genus two curves $H_{\pm x}$, then its compactified Jacobian is a $\PP^1$-bundle over the product $\mc{J}^2(H_x)\times\mc{J}^2(H_{-x})$.

The structure of the fibre $\pi^{-1}(C)$ for $C\in D$ is more complicated. A good reference for sheaves on multiple curves is \cite{Dre08}. Note that every such curve $C$ is (the translate of) the double curve of $H$. There are two kinds of pure sheaves on $C$ which (as a sheaf on $A$) have Mukai vector $(0,2H,0)$. Firstly, there are sheaves which are \em concentrated on \em the reduced genus two curve $H$. These are given by semistable rank two vector bundles of first Chern class $2$ on $H$. We denote this space by $\mc{M}_H(2,2)$. It admits a natural fibration $\det\colon \mc{M}_H(2,2)\rightarrow \mc{J}^{2}(H)$ with fibre isomorphic to $\PP^3$. From the detailed description in \cite{NR69} we see that these fibres themselves are actually isomorphic to the linear system $|2H|$.\\
The second kind of sheaves are extensions of line bundles on the double curve $C$: If $\mc{L}$ is in $\mc{J}^2(H)$, then every element in $E_\mc{L}:=\PP\Ext^1_C(\mc{L},\mc{L}\otimes K_H^\vee)$ defines a semistable sheaf on the double curve $C$ with Mukai vector $(0,2H,0)$. Note that $\dim \Ext^1_C(\mc{L},\mc{L}\otimes K_H^\vee)=4$ and it contains as a codimension one subspace the extensions $\Ext^1_H(\mc{L},\mc{L}\otimes K_H^\vee)$ (as line bundles on $H$) which are, of course, sheaves of the first kind (concentrated on $H$). We obtain a fibre space $\phi_E\colon E\rightarrow \mc{J}^2(H)$ with fibres $E_\mc{L}$. It also admits a map $\det\colon E\rightarrow \mc{J}^2(H)$ which factors as $(-\otimes K_H^\vee)\circ (-)^{\otimes 2}\circ\phi_E.$\\
Altogether we see that $\pi^{-1}(C)$ has the structure of a fibre space over $\mc{J}^2(H)$ with the fibres being the union of $1+16$ $\PP^3$s where each of the $16$ meets the remaining one in a plane.

Let us finish this section by analysing the Albanese fibres of the Jacobians of nodal curves:

\begin{lem}\label{lem:conn_comp}
Let $C$ be a curve in $S$ (resp. $N(1)$, $N(2)$, $N(3)$), then the Albanese fibre of $\pi^{-1}(C)$ has one (resp. one, two, four) connected components.
\end{lem}

\begin{proof}
We will use the natural construction given by the Kummer involution. Let $C$ be a curve inside $S\cup N(1)\cup N(2)\cup N(3)$ and let $f(C)$ be its image under the quotient map $f\colon A\rightarrow Kum_s(A)$. Let $\widetilde{C}$ and $\widetilde{f(C)}$ be the normalisations with induced degree two map $\tilde{f}\colon\widetilde{C}\rightarrow\widetilde{f(C)}$. This map ramifies in $2\times(\text{number of nodes of }C)$ points. It naturally induces a map $\tilde{f}^*$ sending $J(\widetilde{f(C)})$ in $J(\widetilde{C})$, with image $Y$. Let $Z$ be the kernel of the dual map $J(\widetilde{C})\rightarrow J(\widetilde{f(C)})$.\\
Now, let $a$ be the Albanese map $J(C)\rightarrow A$. This map is trivial on the $\mathbb{C}^{*}$ bundle structure, hence to understand the structure of $a^{-1}(0)$ we can work on $\widetilde{C}$ and its induced Albanese map, which we still call $a$. Since the composition of $a$ with $\tilde{f}^*$ on $J(\widetilde{f(C)})$ is trivial, the subvariety $Y$ satisfies $a(Y)=0$. Therefore the Stein factorisation of $a\colon J(\widetilde{C})\rightarrow A$ is given by $J(\widetilde{C})\rightarrow Z'$ composed with an isogeny $Z'\rightarrow A$. Therefore, the number of connected components of $a^{-1}(0)$ equals the degree of this isogeny. Notice that here $Z'$ is given by $Z/(Y\cap Z)$, as all these points are sent to $0$. Therefore the number of connected components of $a^{-1}(0)$ equals the degree of the map $Z\rightarrow A$ divided by the cardinality of $Y\cap Z$.\\
A special case of this setting is analysed in \cite[Theorem 12.3.3]{BL92}: if $C\in S\cup N(1)$, the variety $Z$ is principally polarised and the map $Z\rightarrow A$ has degree $16$. Moreover, $Y\cap Z=Z[2]$, therefore $a^{-1}(0)$ consists of a single connected component. If $C$ lies in $N(2)$, $Z$ has a polarisation of type $(1,2)$ by \cite[Corollary 12.1.5 and Lemma 12.3.1]{BL92} and the map $b\colon Z\rightarrow A$  satisfies $\Theta_{\widetilde{C}}|_Z=2\Theta_{Z}=b^{*}(\Theta_A)$, therefore it has degree $8$. In this case, we have $Y\cap Z=Y[2]$ (since $\widetilde{f(C)}$ is an elliptic curve), therefore $a$ has degree $2$. The final case, when $C\in N(3)$, was already analysed in \cite[Proof of Prop.\ 2.1.4]{Rap07}. Here $Y$ is a point and the map $a$ has degree $4$.
\end{proof}

\section{The Six-Dimensional Case}\label{sec:og6}

Now we return to the original question about automorphisms on manifolds of $Og_6$-type. We start by the following observation which shows that the situation is somewhat more complicated than in the ten-dimensional case.
\begin{lem}
Let $G_0$ be the group generated by points of order $2$ in $A\times A^*$. Then we have an induced action of $G_0$ on $\widetilde{K}(2v)$ which acts trivially on $H^2$.
\end{lem}
\begin{proof}
For any Mukai vector $v=(r,l,a)$ we have an induced action of $A\times A^*$ on $M(v)$. It has been described by Yoshioka in \cite{Yos99} (cf.\ diagram (1.8)). First he defines the following map
\begin{eqnarray*}
\tau_v\colon A\times A^*&\rightarrow & A\times A^*,\\
(x,\mc{L})&\mapsto & (x',\mc{L}'):=(rx-\hat{\phi}_l(\mc{L}),-\phi_l(x)-a\mc{L}),
\end{eqnarray*}
where $\phi_l\colon A\rightarrow A^*$ and $\hat{\phi}_l\colon A^*\rightarrow A$ are defined as usual (e.g. $\phi_l(x):=t_x^*\mc{N}\otimes\mc{N}^\vee$ for some $\mc{N}$ with $c_1(\mc{N})=l$). Yoshioka then defines an action on $M(v)$ as follows:
\begin{eqnarray*}
\Phi\colon A\times A^*\times M(v)&\rightarrow & M(v),\\
(x,\mc{L},\mc{F})&\mapsto & t_{x'}^*\mc{F}\otimes \mc{L}'.
\end{eqnarray*}
The introduction of $\tau_v$ has the advantage that
\[ alb(\Phi(x,\mc{L},\mc{F}))=alb(\mc{F})+(nx,\mc{L}^{\otimes n}),\]
where $n:=l^2/2-ra=v^2/2.$

Now, the crucial point in our situation is that since our Mukai vector is non-primitive, we have $\tau_{2v}=2\tau_v$ and we define an action $\Phi'$ as above on $M(2v)$ using $\tau_v$ instead of $\tau_{2v}$. Since $(2v)^2/2=4,$ we see that $alb(\Phi(x,\mc{L},\mc{F}))=alb(\mc{F})+(2x,\mc{L}^{\otimes 2})$ and can immediately deduce that the action of $G_0$ preserves the Albanese fibres if $x$ and $\mc{L}$ are two-torsion. The action on $H^2(K(2v),\Z)$ can be computed via Lemma 1.34 of \cite{MW14} and is easily seen to be trivial. The group $G_0$ certainly preserves the singular locus of both $M(2v)$ and $K(2v)$ and the action extends naturally to an action on the desingularisation (cf.\ the description of the normal bundle of the singular locus in \cite[Prop.\ 4.3]{MW14}.
\end{proof}

Thus we have $G_0\subseteq \ker\nu$. The converse is also true:
\begin{thm}\label{thm:thm_og6}
Let $X$ be a manifold of $Og_6$-type. Then the kernel of the cohomological representation
\[\nu\colon \aut(X)\rightarrow O(H^2(X,\Z))\]
is isomorphic to $G_0:=\langle A[2], A^*[2]\rangle\cong (\Z/2\Z)^{\times8}.$
\end{thm}

\begin{proof}
We consider the manifold $X$ which is obtained as the resolution of the Albanese fibre of the relative compactified Jacobian $\mc{J}^4(|2H|)$ from the previous section. 

Let $\psi$ be an automorphism of $X$ acting trivially on cohomology. Let us prove $\psi\in G_0$. Again, $\psi$ descends along the blow down to an automorphism of $\mc{K}^4(|2H|)$ which we will denote by the same symbol. Also, the fibration $\pi\colon \mc{K}^4(|2H|)\rightarrow |2H|$ is $\psi$-equivariant and this time we can prove directly that (up to the action of $A[2]$) $\psi$ is, in fact, preserving the fibres of $\pi$: The detailed analysis of the fibres of $\pi$ in the previous section show that $\psi$ must preserve the stratification of $|2H|$.

\begin{lem}
Any automorphism of $|2H|\cong \PP^3$ preserving its stratification (by analytical type of the singularities of the curves) is, in fact, induced by translation of a point in $A[2]$.
\end{lem}
\begin{proof}
Any such automorphism induces an automorphism of the closure the stratum $R(1)$ which is isomorphic to $Kum_s(A)$. Furthermore $D$ corresponds to its $16$ nodes. Thus we deduce that we get an automorphism of $Kum_s$ preserving the set of nodes. Such an automorphisms can be lifted to the abelian surface $A$ and has to act there as translation by a two-torsion point. 
\end{proof}


Composing with the translation of an appropriate element in $A[2]$, we may thus assume that the action on $|2H|$ is trivial, that is, for all generic smooth $C\in|2H|$ we obtain an automorphism of $\mc{K}^4(C)$. We continue by the analogue of \Ref{lem}{lem_og10_rigid} 
\begin{lem}
The restriction $\mc{D}:=\Theta\cap \mc{K}(2v)$ of the relative Theta divisor to $\mc{K}(2v)$ is an effective rigid divisor.
\end{lem}
\begin{proof}
On a general fibre $\mc{K}^4(C),$ the divisor $\mc{D}$ is given by
\[ D_C=i^*\Theta_C=\{ \mc{O}(p+\iota(p)+q+\iota(q))\mid p,q \in C\}.\]
We can thus define a rational map
\begin{eqnarray*}
\mc{D}&\dashedrightarrow& Sym^2(Kum_s),\\
\mc{O}(p+\iota(p)+q+\iota(q))&\mapsto & p+q.
\end{eqnarray*}
The fibre over $p+q$ consists of curves in $|2H|$ that pass through $p$ and $q$, hence is a $\PP^1$. Thus, by \Ref{lem}{lem_rigid} we deduce that $\mc{D}$ is, in fact, rigid.
\end{proof}

\begin{rmk}
A divisor similar to $\mc{D}$ above appeares in the Main Theorem of \cite{Nag13}.
\end{rmk}

Thus (up to the action of $A[2]$) any automorphism $\varphi$ in $\ker\nu$ induces a non-trivial automorphism $\varphi$ of $\mc{K}^4(C)$ preserving a divisor in $i^*|\Theta_C|$. 

Note that our automorphism $\varphi$ cannot be given by $-\id$ on a generic fibre $\mc{K}^4(C)$ because this would yield a non-symplectic automorphism of $\widetilde{K}^4(|2H|)$ (which would not act trivially on the second cohomology). Furthermore, since $\varphi$ is of finite order (\Ref{prop}{prop_huy}) and generically $\mc{K}^4(C)$ is a simple abelian variety, it then has to be given by the translation $t_x$ by a point $x\in\mc{K}^4(C)$ of finite order. By the lemma above, $t_x$ preserves a divisor in $|i^*\Theta_C|$. Hence $t_x^*\mc{O}(i^*\Theta_C)\cong\mc{O}(i^*\Theta_C)$ and therefore $x$ is in the kernel of the map $\phi_{i^*\Theta_C}$ associated with the polarisation (usually denoted by $K(i^*\Theta_C)$). But it is well-known (cf.\ \cite[Sect.\ 2]{LP09}) that 
\[K(i^*\Theta_C)\cong A\cap\mc{K}^4(C)=\{x=-x=\iota^*x=-\iota^*x\}=A[2].\qedhere\]
\end{proof}


\section{The Fixed Locus}\label{sec:fixed}
Let us end the discussion by studying the automorphisms in $\ker\nu$ in more detail. We start by analysing automorphisms in the subgroup $A[2]$.

\begin{prop}\label{prop:fixed_ta}
Let $a\in A[2]\setminus{\{0\}}$ be a two-torsion point and $X$ a manifold of $Og_6$-type. Then the fixed locus of the induced action of $t_a$ on $X$ as described above consists of the disjoint union of $16$ K3 surfaces.
\end{prop}
\begin{proof}
The automorphism $t_a$ deforms with all deformations of $X$ and its fixed locus is a deformation invariant. Therefore we might assume that $X$ is given as the desingularisation of the Albanese fibre $\mc{K}^4(|2H|)$ of the relative compactified Jacobian (of degree four and genus five) as in the last sections. Now, let us start by analysing the fixed locus of the action of $t_a$ on $\mc{J}^4(|2H|)$. We follow the ideas of Oguiso in the case of generalised Kummer varieties (cf.\ \cite[Prop.\ 3.6]{Ogi12}). A sheaf $\mc{F}$ on $A$ is fixed by $t_a$ if and only if it is a pullback from the quotient $A/\langle a\rangle$. Thus the fixed locus of the Jacobian $\mc{J}^4(|2H|)$ is isomorphic to the Jacobian of degree two over the quotient linear systems (of genus three) on $A/\langle a\rangle$. These linear systems correspond to the fixed point locus of the induced action of $t_a$ on $|2H|$. The restriction of this action to the dual singular Kummer is, of course, just given by translation of a two-torsion point. In particular it has (on the Kummer) precisely $8$ fixed points. Thus we conclude that the fixed locus of $|2H|$ consists of two distinct lines, say $l_1$ and $l_2$.

Now we intersect with the Albanese fibre $\mc{K}^4(|2H|)$. Observe that a sheaf of the form $q_a^*\mc{G}$ (where $q_a\colon A\rightarrow A/\langle a\rangle$ denotes the quotient) is in the Albanese fibre over $0$ if and only if $\mc{G}$ is in the Albanese fibre over a point in $q_a(A[2])\cong A[2]/\langle a\rangle$. (Note that $|A[2]/\langle a\rangle|=8$.) These fibres are surely all isomorphic. Thus we consider only the fibre over $0$. The Albanese fibre of the Jacobian over a quotient linear system ($l_1$ or $l_2$) is a generalised Kummer variety of dimension two. In particular, it is an elliptic K3 surface fibred over the linear system. 
\end{proof}

\begin{rmk}
Let us study the fibre structure of the elliptic fibration above in more detail. The special fibres correspond to the intersections of $l_1$ (and $l_2$ resp.) with the strata $N(2)$ and $R(1)$. (Note that for symmetry reasons there is no intersection with $N(1)$ and $N(3)$ and furthermore we have no intersection with $D$ because $D$ corresponds to the nodes of the dual Kummer and $t_a$ acts transitively on the set of nodes.) Now, if we fix a node $y\in A[2]$, then the line in $N(2)$ corresponding to curves with nodes both in $y$ and $y+a$ is preserved by $t_a$ and we have two fixed points corresponding to the intersections with $l_1$ and $l_2$. There are eight such lines. Now, the fixed locus is isomorphic to the Jacobian of the quotient curve, i.e.\ a curve of arithmetic genus three with one node. Its non-compactified Jacobian is a $\C^*$-bundle. By \Ref{lem}{conn_comp} the Albanese fibre of the Jacobian of a curve in $N(2)$ has two connected components. The compactification of the Jacobian glues the two copies of $\C^*$ two a single $I_2$-fibre. Thus we altogether obtain $8$ $I_2$-fibres. This settles the intersection of $N(2)$ and $l_i$.\\
The stratum $R(1)$ is isomorphic to the singular quartic Kummer surface associated with $A$ and the action $t_a$ induces a symplectic involution. It must therefore have eight fixed points. They correspond to the intersection points of $R(1)$ with $l_1$ (and $l_2$; four points each). These intersection points correspond to curves of the form $H_x+H_{-x}$ where $2x=a$. This time the quotient curve is isomorphic to $H$ with two points joined to a node (the images of the intersection $H_x\cap H_{-x}$). Thus its (non-compactified) Jacobian is a $\mathbb{C}^*$-bundle over $\mc{J}^2(H)$. Now, the boundary of this Jacobian is contained in the singular locus of $\mc{J}^4(|2H|)$ and thus each point is replaced by a $\PP^1$ when passing to the resolution. We see that the special fibre in this case is a $I_2$. (The $\mathbb{C}^*$ is compactified to a $\PP^1$ which meets the exceptional curve in two points.)

Thus we conclude that all the $16$ fixed K3 surfaces are, in fact, elliptic K3s with $12$ $I_2$-fibres.
\end{rmk}

Next, we consider automorphisms in the subgroup $A^*[2]$:

\begin{prop}\label{prop:fixed_l}
Let $\mc{L}\in A^*[2]$ be a non-trivial two-torsion line bundle on $A$ and $X$ a manifold of $Og_6$-type. Then the fixed point locus of the induced action of $\mc{L}$ on $X$ is the disjoint union of $16$ K3 surfaces.
\end{prop}
\begin{proof}
Again we will start by analysing the induced action on the singular Jacobian $\mc{J}^4(|2H|)$. Let $i\colon C\hookrightarrow A$ be a curve in $|2H|$. If $i_*\mc{F}$ is a sheaf in $M(2v)$ with support $C$ then the action of $\mc{L}$ maps $i_*\mc{F}$ to $i_*\mc{F}\otimes\mc{L}\cong i_*(\mc{F}\otimes i^*\mc{L}).$ We deduce that $\mc{L}$ acts fibrewise on $\pi\colon \mc{J}^4(|2H|)\rightarrow |2H|$. For any $C\in|2H|$ the sheaf $i^*\mc{L}$ is a two-torsion line bundle, which is never trivial by the next lemma:

\begin{lem}\label{lem:injectivity}
Let $i\colon C\hookrightarrow A$ be a curve in $|2H|$, then the group homomorphism $i^*\colon \mc{J}^0(A)\rightarrow \mc{J}^0(C)$ is injective.
\end{lem}
\begin{proof}
We are indebted to H.\ Ohashi for filling this gap in our considerations. It is enough to prove that for any $\mc{L}\in \mc{J}^0(A)\setminus \mc{O}_A$ its restriction $i^*\mc{L}$ has no global section. Thus we look at the exact sequence
\[ 0 \rightarrow \mc{L}(-C) \rightarrow \mc{L} \rightarrow i^*\mc{L} \rightarrow 0.\]
Since $\mc{L}$ has no cohomology it is enough to prove that $h^1(\mc{L}(-C))$ vanishes. But now $[\mc{L}(-C)]=[-C]=-2H$ is the negative of an ample class on an abelian surface, thus $\mc{L}(-C)$ has only one non-vanishing cohomology class. But its euler characteristic is four.
\end{proof}

Thus we conclude that for any smooth curve $C\in|2H|$ the action of $\mc{L}$ on the fibre $\pi^{-1}(C)$ is transitive. Furthermore we can show that there is no fixed point in the strata $R(1)$ and $R(2)$. Indeed, the Jacobian of a curve $H_x+H_{-x}$ in $R(1)\cup R(2)$ is a $\PP^1$-bundle over $\mc{J}^2(H_x)\times \mc{J}^2(H_{-x})$ and it is easy to see that the action of $\mc{L}$ on the base of this bundle is transitive.

Let us continue with the stratum $D$. As described in the last section, the compactified Jacobian of the curve $C$ (double the curve $H$) is a fibrespace over $\mc{J}^2(H)$. Recall that the fibre map is given by the determinant. Thus the induced action on $\mc{J}^2(H)$ is given by $\mc{M}\cdot \mc{L} = \mc{L}^2\otimes \mc{M}$ (for any $\mc{M}\in \mc{J}^2(H)$). Thus $\mc{L}$ acts trivially on the base. The fibre of the Jacobian over a fixed $\mc{M}\in\mc{J}^2(H)$ consists of $16+1$ $\PP^3$s. Each of the $16$ $\PP^3$s is given by extensions of the form $\Ext^1(\mc{N},\mc{N}\otimes K_H^\vee)$ (for $\mc{N}$ satisfying $\mc{N}^{\otimes2}\otimes K_H^\vee\cong \mc{M}$), thus we see that the action of $\mc{L}$ acts transitively on the set of these $16$ $\PP^3$s. We are left with studying the remaining $\PP^3$ parametrising sheaves concentrated on $H$. This $\PP^3$ is equivariantly isomorphic to $|2H|$, where the action on the latter is induced by translation by a two-torsion point. Thus we see that the fixed locus consists of two distinct lines. Now, each line meets the stratum $R(1)$ in four points. These points belong to the singular locus of $\mc{J}^4(|2H|)$ and when we pass to its resolution, they are replaced by a $\PP^1$. Altogether we see that for every $C\in D$ the fixed locus in the fibre $\pi^{-1}(C)$ consists of two distinct $I_0^*$-configurations of lines (as in the Kodaira classification).

If $C$ is a stable curve in $N(i)$ with nodes $p_1,\dots,p_i$, then the fibre $\pi^{-1}(C)$ (the compactified Jacobian) has a stratification corresponding to its partial normalisations at the nodes. Each stratum can be identified with a $(\mathbb{C}^*)^{\times n}$-bundle ($0\leq n\leq i$) over the Jacobian of the normalisation $\tilde{C}$.\\
Now, the action of $\mc{L}$ is induced via the natural action of the degree zero (non-compactified) Jacobian of $C$ (or its partial normalisations). This action is always given as the translation by a two-torsion line bundle on the base of the $(\mathbb{C}^*)^{\times n}$-bundle and by $\pm1$ on the $\mathbb{C}^*$-fibres. It is very difficult to determine whether this two-torsion line bundle is trivial or not. If it is non-trivial, we immediately deduce that there is no fixed locus. If the two-torsion line bundle is trivial, we see that the fixed locus consists of an abelian $(5-i)$-fold and possibly some $(\mathbb{C}^*)^{\times n}$-bundles, depending on the action on the $\mathbb{C}^*$-fibres. 

\begin{lem}
Every element of $A^*[2]$ acts transitively or trivially on the set of connected components of the Albanese fibre of the Jacobian of a curve in $N(i)$.
\begin{proof}
An element of $A^*[2]$ acts either as a non-trivial translation on the abelian part of the compactified Jacobian or purely on the non-abelian part. In the first case, there are no fixed points for the action and it is therefore transitive on the connected components of the Albanese fibre, while in the second case it acts trivially on the base an therefore also trivially on the connected components of the Albanese fibre. Note that the number of connected components was computed in \Ref{lem}{conn_comp}. 
\end{proof}
\end{lem}

Let us use the detailed description of the $D$ stratum to proceed. The fixed locus in the fibre of a curve in $D$ has dimension three. Since, at the end, the fixed locus of the resolution of the Albanese fibre $\mc{K}^4(|2H|)$ has to be smooth symplectic, we can firstly deduce that there is no curve in $C\in N(1)$ with trivial action on the Jacobian of its resolution: Indeed, this would yield a two-dimensional family of fixed abelian surfaces which cannot degenerate to the $I_0^*$ in the $D$ stratum. We continue with curves in the $N(2)$ stratum. Since the Albanese fibre of the Jacobian of such a curve has two connected components, we see that for each point $a\in D$, there must be precisely one line in $N(2)$ passing through $a$ (and another point $a'\in D$) containing curves such that the pullback of $\mc{L}$ to the normalisation is trivial. Thus there are exactly $8$ such lines. Let us denote one of them by $l$. The corresponding action on the two $(\mathbb{C}^*)^2$ has to be given by $(-1,-1)$. Thus this part of the fixed locus consists of two disjoint elliptic K3 surfaces, each with two $I_0^*$-fibres. Furthermore it has six $I_2$-fibres, which correspond to the intersections of $l$ with the stratum $N(3)$. Indeed, there are six such intersection points and the Albanese fibre of the corresponding (non-compactified) Jacobian has four connected components each of which is a $\mathbb{C}^*$. In the Compactification they are glued to two $I_2$-fibres. Note that on an intersection point of $l$ with $N(3)$ the action on the $(\mathbb{C}^*)^3$-fibres is given by $(-1,-1,+1)$. To conclude the proof, the following computation suffices:

\begin{lem}
Let $\Gamma$ be a connected component of the fixed locus of $A^*[2]$ over $N(1)$ (respectively, $N(2)$ or $N(3)$). Then it is fixed by $1$ (resp. $2$, $4$) elements of $A^*[2]$.
\begin{proof}
We already proved that only the identity has a fixed locus over $N(1)$. For what concerns $N(2)$, we proved that every non-trivial element of $A^*[2]$ fixes exactly $16$ components over $8$ lines in $N(2)$. Since there are $15$ non-trivial elements in $A^*[2]$ and $120$ lines, to conclude we only need to prove that every line supports fixed points of at most one non-trivial element. Indeed, as we have seen in \Ref{lem}{conn_comp}, the Albanese fibre of the normalisation $J(\widetilde{C})$ has two connected components each of which is isomorphic to the image $Y$ of the Jacobian $J(\widetilde{f(C)})$, which is an elliptic curve. The action of $A^*[2]$ acts thus by a subgroup of order $8$ on this constellation, leaving an order $2$ stabiliser, which consists of the identity and the unique non-trivial element.\\
In the $N(3)$-case the Albanese fibre consists of $4$ points and we thus have a stabiliser of order $4$. 
\end{proof}
\end{lem}

The above lemma tells us, in particular, that we have exactly $48=\frac{240}{15/3}$
elements of $N(3)$ which have a nonempty fixed locus for a given element of $A^*[2]$. (Here $240$ is the number of points in $N(3)$, $15$ is the number of all non-trivial involutions and $3$ is the number of involutions acting non-trivially on the fibre of one point in $N(3)$.) These are precisely the $48$ points of $N(3)$ lying in the $8$ fixed lines of $N(2)$, therefore we have no isolated fixed points.
\end{proof}



Finally, let consider the remaining 'mixed' automorphisms.

\begin{prop} \label{prop:fixed_mixed}
Let $\varphi$ be an element of $(A\times A^*)[2]\setminus(A[2]\cup A^*[2])$. Then the fixed point locus of its action on $X$ consists either of $16$ fixed points ($180$ cases) or of $2$ fixed $K3$ surfaces ($45$ cases).
\begin{proof}
Let $\varphi$ be $t_a\circ \mc{L}$ with $a\in A[2]$ and $\mc{L}\in A^*[2]$. The fixed locus of $\varphi$ is given by points $x$ such that $t_a(x)=\mc{L}(x)$. As $t_a$ acts on the linear system and $\mc{L}$ acts on the Jacobians, this is only possible over the two lines of $\mathbb{P}^3$ fixed by $t_a$. These lines are the intersection with $|2H|$ of the pullback of $\{P\}$ from $A/\langle a\rangle$, where $P$ is the $(1,2)$-polarisation on $A/\langle a\rangle$. Let $C$ be a smooth curve corresponding to a point on one of these two lines and let $C/a$ be the quotient curve in $A/\langle a\rangle$. Let $X:=\mc{J}^4(C)$, let $Y\subset X$ be the image of $\mc{J}^2(C/a)$ under the pullback of the covering map, and let $Z$ be the complementary abelian surface inside $X$. Via pullpack along the embedding of $C$, $A^*$ is embedded into $Y$. Thus, $\mc{L}$ acts on $X$ through a two-torsion point $y_0\in Y$. The involution $t_a$ acts as the identity on $Y$ and as $-1$ on $Z$. We have an exact sequence
\[0\rightarrow Z[2]\rightarrow Z\times Y\rightarrow X\rightarrow 0.\]
Our fixed locus is given by pairs $(z,y)$ such that $(-z,y)-(z,y+y_0)$ lies in $Z[2]$. If $y_0\notin Z[2]$, this has no solution, otherwise there is a threefold isomorphic to $Y$ which is fixed by $\varphi$ inside $X$, and the Albanese fibre is an elliptic curve in this threefold. (This fibre is connected by part a) of the lemma below.) In particular, we have a group homomorphism $A^*[2]\rightarrow Y[2]/Z[2]$ induced by $C\subset A$ and the above exact sequence. By part b) of the Lemma below, this map is surjective. Hence, there are three non-trivial elements of $A^*[2]$ lying in $Z[2]$. The fixed elliptic curve inside $X$ deforms with $C$, giving an elliptic $K3$ surface fixed by $\varphi$ over both lines in $|2H|$ fixed by $t_a$. On the other hand, if $\mc{L}$ does not map inside $Z[2]$, there is no fixed locus over all smooth curves fixed by $t_a$ and, with an analogous argument, this holds also for curves lying in $N(2)$. The only case left is given by curves lying in $R(1)$. For such a curve the Jacobian is a $\mathbb{P}^1$-bundle over the product of the Jacobians of the two irreducible components of the curve. The map $t_a$ acts by exchanging the two factors and $\mc{L}$ acts transitively on each factor of the base. Hence, the set $(x,\mc{L}(t_a(x)))$ parametrises a surface of fibres fixed by $\varphi$, and in every fibre we have two fixed points. After taking the Albanese fibre, we are left with two fixed points for every $t_a$-fixed curve in $R(1)$, which sums to $16$.
\end{proof} 
\end{prop}
\begin{lem}
Keep notations as in the above proposition.
\begin{enumerate}[a)]
\item If we have a fixed threefold $Y$, the Albanese fibre of $Y$ is connected.
\item The map $A[2]\rightarrow Y[2]/Z[2]$ is surjective.
\end{enumerate}

\begin{proof}
Part a): We need to analyse the number of connected components of the kernel of the map $Y\rightarrow A$. Let us call this kernel $D$. The threefold $Y$ admits a two-to-one covering by the Jacobian $\mc{J}^2(C/t_a)$ of the quotient curve $C/t_a$. We pull back along this covering to obtain the following diagram:
\[\xymatrix{
\Z/2\Z \ar[r]^{\cong} \ar@{^(->}[d] & \Z/2\Z \ar@{^(->}[d] \\
\widetilde{D} \ar@{^(->}[r]  \ar@{->>}[d]^{2:1} & \mc{J}^2(C/t_a) \ar@{->>}[r] \ar@{->>}[d]^{2:1} & A \ar[d]^{\cong}\\
D \ar@{^(->}[r] & Y  \ar@{->>}[r] & A.
}\]

Thus we see that it is enough to show that the double cover $\widetilde{D}$ is connected. This is true because the polarisations of both $\mc{J}^2(C/t_a)$ and $A$ are principal and thus $A$ and $\widetilde{D}$ have exponent $1$ as subvarieties of $\mc{J}^2(C/t_a)$.

Part b): If the curve $C$ is general, all maps $A\rightarrow Y$ are obtained from the inclusion map $A\hookrightarrow Y$ by composing with a multiplication map. Consider the following diagram:
\[\xymatrix{
Z \ar@{^{(}->}[r] & X \ar@{->>}[r] & X/Z \ar@{->>}[r]^{4:1} & Y\\
A\cap Z \ar@{^{(}->}[r] \ar@{^{(}->}[u] & A \ar[r] \ar@{^{(}->}[u] & A/(A\cap Z). \ar@{^{(}->}[u]
}\]
Here the first line is the dual of the exact sequence $Y\rightarrow X\rightarrow Z$ and the last map is four to one. The bottom line is exact and we have a composition map $$ A\rightarrow A/(A\cap Z) \rightarrow X/Z \rightarrow Y$$ which must be the composition of the inclusion $A\rightarrow Y$ with the multiplication by two (here we use that $Y\cap Z$ consists only of two torsion points). This means that the map $A\rightarrow A/A\cap Z$ has degree four, so $A\cap Z$ consists of four elements. The group homomorphism $A[2]\rightarrow Y[2]/Z[2]$ sends only the four elements of $A\cap Z$ to zero, hence it is surjective.
\end{proof}
\end{lem}

\begin{oss}
The above computations of the fixed locus tells us also that the action of these automorphisms on the full cohomology group is non trivial. Indeed, the equivariant cohomology is the cohomology of the quotient manifold, which has topological Euler characteristic different from $Og_6$.
\end{oss}

\end{document}